\documentclass{article}
\usepackage[utf8]{inputenc}
\usepackage{amsfonts}
\usepackage{geometry}
\usepackage{amsthm}
\usepackage{amsmath}
\usepackage{setspace}
\usepackage{array}
\usepackage[T1]{fontenc}

\newcommand{\R}{\mathbb{R}}

\newtheorem{thm}[subsection]{Theorem}
\newtheorem{lem}[subsection]{Lemma}
\newtheorem{cor}[subsection]{Corollary}
 
\newtheorem{rem}[subsection]{Remark}

\title{Local Stability of Einstein Metrics Under the Ricci Iteration}
\author{Timothy Buttsworth, Maximilien Hallgren }
\date{}

\begin{document}

\maketitle
\begin{abstract}
        We provide a sufficient condition for the local stability of closed Einstein manifolds of positive Ricci curvature under the Ricci iteration in terms of the spectrum of the Lichnerowicz Laplacian acting on divergence-free tensor fields. We use this result to consider the stability of several Einstein manifolds under the Ricci iteration, including symmetric spaces of compact type. 
\end{abstract}
\section{Introduction}
On a smooth manifold $M$, a Ricci iteration is defined to be a sequence of Riemannian metrics $\{g_i\}_{i=1}^{\infty}$ such that 
\begin{equation}\label{RI}
Ric(g_{i+1})=g_i, \ i\in \mathbb{N},
\end{equation}
where $Ric(g)$ is the Ricci curvature of $g$. 
The Ricci iteration was introduced by Rubinstein as a method of discretising the Ricci flow 
\begin{equation}\label{RF}
\frac{\partial g}{\partial t}=-2Ric(g),
\end{equation} 
and has several applications in K\"ahler geometry \cite{Rubinstein07,Rubinstein08}. 
If $(M,g_1)$ is K\"ahler, a Ricci iteration exists if and only if
a positive multiple of $g_1$ represents the first Chern class \cite{DarvasRubinstein19}. Furthermore, the iteration
then converges modulo diffeomorphisms to a K\"ahler-Einstein metric whenever one exists. 
The non-K\"ahler case has also been studied, but only for a limited class of homogeneous spaces \cite{BPRZ18,PulemotovRubinstein}. 

A prerequisite to studying the Ricci iteration is an understanding of solutions to the prescribed Ricci curvature problem 
\begin{equation}\label{PRCE}
Ric(g)=T,
\end{equation}
because a Ricci iteration is constructed by recursively inverting the Ricci curvature operator. The problem of solving \eqref{PRCE} has been widely studied in, for example, \cite{Delanoe03,Delay01,Delay02,DeTurck82,DeTurck85,HMPRC}. One difficulty in proving existence of solutions to \eqref{PRCE} is the diffeomorphism invariance of the Ricci curvature, which manifests itself in the non-ellipticity of \eqref{PRCE}. The issues arising from the non-ellipticity of the Ricci curvature are also encountered when establishing short time existence results for solutions of the Ricci flow because the non-ellipticity of the Ricci curvature implies that \eqref{RF} is not a parabolic equation. 
The celebrated `DeTurck trick' technique originating from \cite{DeTurck82,DeTurck83} overcomes the non-parabolicity  of \eqref{RF} by replacing it with an equivalent parabolic flow, which evolves a metric by its Ricci curvature as well as by a diffeomorphism. 

The purpose of this paper is to examine the dynamical stability of the Ricci iteration \eqref{RI} near a fixed point of the iteration, i.e., a closed Einstein manifold $(M,g)$ whose Einstein constant is $1$. In this context, solutions to the prescribed Ricci curvature problem are well-understood by use of the Inverse Function Theorem for Banach spaces, as seen by DeTurck in \cite{DeTurck85}. To study the Ricci iteration, we also use the Inverse Function Theorem, but our key technique is developing a certain local and discrete analogue of the DeTurck trick to overcome the obstacle presented by the diffeomorphism invariance of the Ricci curvature. The result we obtain is similar to Theorems 1.4 and 1.5 in \cite{Kro13}, which relate dynamical stability of the Ricci flow near an Einstein manifold to both the spectrum of the rough Laplacian and whether or not the Einstein manifold maximises the Yamabe functional. 

This paper is organised as follows. In 
Section 2, we set up notation and appropriate Hilbert spaces on which to study the problem of inverting Ricci curvature. In Section 3, we state and prove Theorem \ref{LSTR}, which relates dynamical stability of fixed points of the Ricci iteration to the spectrum of the Lichnerowicz Laplacian acting on divergence-free $(0,2)$-tensor fields. In Section 4, we review standard material which allows us to compute the spectrum of the Lichnerowicz Laplacian on compact symmetric spaces. Finally, in Section 5, we look at specific examples of Einstein metrics, and determine whether they satisfy the assumptions of Theorem \ref{LSTR}. 
\section{Preliminaries and Notation}
Consider a closed Riemannian manifold $(M,g)$ satisfying $Ric(g)=g$. 
We denote by $T^* M$ the set of smooth one-forms on $M$, and by $S^2 T^* M$ the set of smooth symmetric $(0,2)$ tensor fields on $M$. 
It is well-known (see Proposition 2.3.7 of \cite{Topping}) that the linearisation of the Ricci curvature operator applied to $h\in S^2 T^* M$ is given by
\begin{align*}
  Ric(g)'(h)=-\frac{1}{2}\Delta_L h-\delta^* \delta G(h).  
\end{align*}
Here, $\Delta_L:S^2T^* M\to S^2 T^* M$ is the Lichnerowicz Laplacian, $\delta:S^2 T^* M\to T^* M$ is the divergence operator, $\delta^*:T^* M\to S^2 T^* M$ is the symmetrised covariant derivative, and $G:S^2T^* M\to S^2 T^* M$ is the gravitation tensor. If $h\in S^2 T^* M$, $v\in T^* M$, and $X,Y$ are vector fields on $M$, these operators can be defined by 
\begin{align*}
\Delta_L(h)(X,Y)&=(\Delta h)(X,Y)-h(X,Ric(Y))-h(Ric(X),Y)+2tr^g ( h(R(X,\cdot)Y,\cdot)),\\
\delta(h)(X)&=-tr^g ((\nabla_{\cdot} h)(X, \cdot)),\\
\delta^*(v)(X,Y)&=\frac{1}{2}(\nabla_X v(Y)+\nabla_Y v(X)),\\
G(h)(X,Y)&=h(X,Y)-\frac{1}{2}(tr^g h)g(X,Y),
\end{align*}
where $\Delta:S^2 T^* M\to S^2 T^* M$ is the rough Laplacian. 
By a computation involving the commuting of derivatives, we can show that both $\mbox{im}(\delta^{\ast})$ and $\ker \delta$ are invariant under the action of $\Delta_L$.  

We will study the above operators on $H^k(S^2 T^*M)$ by which we mean the completion of $S^2 T^* M$ with respect to the Sobolev norm $H^k$ induced by $g$. Similarly, we define $L^2 (S^2 T^* M)$ to be the completion of $S^2 T^* M$ with respect to the usual $L^2$ norm induced by $g$, and other occurences of $H^k$ and $L^2$ are defined analogously. We equip $H^k(S^2 T^* M)$ with the bilinear form
\begin{align}\label{NIP}
\langle h_1,h_2\rangle_k =\int_M \langle (C-\Delta_L)^{\frac{k}{2}}h_1, (C-\Delta_L)^{\frac{k}{2}}h_2\rangle_g d\mu_g 
\end{align}
for some constant $C > 0$, where $\mu_g$ is the volume form induced by $g$. 
The following lemma demonstrates that for an appropriate choice of $C$, these bilinear forms are defined and turn the collection of vector spaces $\{H^k(S^2 T^*M)\}_{k\in \mathbb{N}}$ into a family of Hilbert spaces.  
\begin{lem}\label{NHSS}
The constant $C$ can be chosen so that, in the sense of quadratic forms on $L^2$, we have
 $$I - \Delta  \leq C - \Delta_L \leq B(I - \Delta)$$ for some constant $B>0$. In this case:
 
(i) The expression $\langle \cdot,\cdot\rangle_k$ in \eqref{NIP} is defined. 

(ii) For each $h\in H^{l}(S^2 T^*M)$, $\langle h,h\rangle_{l}\ge \langle h,h\rangle_k$ whenever $l\ge k$.  

(iii) The form $\langle \cdot,\cdot\rangle_k$ is a scalar product on $H^k(S^2 T^*M)$, and the associated norm is equivalent to the usual $H^k(S^2 T^*M)$ norm. 
\end{lem}
\begin{proof}
Since $A:=\Delta_L-\Delta:L^2(S^2 T^* M)\to L^2(S^2 T^* M)$ is a bounded linear operator, the constant $C$ can be chosen so that $$I - \Delta  \leq C - \Delta_L.$$
On the otherhand, 
\begin{align*}
C-\Delta_L\le C+\left|A\right|_{L^2}-\Delta,
\end{align*}
which can be made less than $B(I -\Delta)$ for large enough $B$. 

For this choice of $C$, conditions $(i)$ and $(ii)$ hold because $C-\Delta_L$ is self-adjoint on $L^2$, and satisfies $C-\Delta_L\ge I>0$. To check condition $(iii)$, note that it suffices to show that the norms are equivalent for smooth tensors. Now 
by arguments appearing in page 363 of \cite{TaylorPDE}, we know that 
$$(h_1, h_2)_k := \int_M \langle (I -\Delta )^{k/2} h_1 , (I -\Delta)^{k/2} h_2 \rangle_g d\mu_g $$
is an inner product for $H^k(S^2 T^{\ast}M)$ giving an equivalent norm to the usual Sobolev norm. Therefore, the proof will be complete if we can demonstrate that $\langle\cdot,\cdot\rangle_k$ is equivalent to $(\cdot,\cdot)_k$ on smooth tensors. This follows from the estimate 
\begin{align*}
(h,h)_k&=\int_M \langle (I -\Delta)^k h, h \rangle_g d\mu_g\\
&\leq \int_M \langle (C - \Delta_L)^k h, h \rangle_g d\mu_g\\
& \leq \int_M B^k\langle (I -\Delta)^k h, h \rangle_g d\mu_g\\
&=B^k (h,h)_k
\end{align*} 
and the observation that $\Delta_L$ is self-adjoint with respect to $L^2(S^2 T^* M)$, so  $\int_M \langle (C - \Delta_L)^k h, h \rangle_g d\mu_g=\langle h,h\rangle_k$. 
\end{proof}
We now have Hilbert spaces on which to study the linearisation of the Ricci curvature. We chose the new inner products because $\Delta_L$ is now self-adjoint on $H^k(S^2 T^*M)$, but also because of the new relationship between $H^{k}(\ker \delta)$ and $\delta^*(H^{k+1}T^* M)$. Indeed, since both $\ker \delta$ and $\mbox{im}(\delta^*)$ are invariant under the action of $\Delta_L$,  $H^{k}(\ker \delta)$ and $\delta^*(H^{k+1}T^* M)$ are now orthogonal in $H^k(S^2 T^* M)$. In fact, since $\delta \delta^*$ is elliptic and non-negative, results appearing in \cite{Berger} imply that \begin{align*}
    H^{k}(S^2 T^* M)=H^{k}(\ker \delta)\oplus \delta^*(H^{k+1}T^* M),
\end{align*} 
so $H^{k}(\ker \delta)$ and $\delta^*(H^{k+1}T^* M)$ are actually orthogonal complements of each other, and are both closed because $H^{k}(\ker \delta)$ is closed. 
For ease of notation, we define $$B_k(r) := \{h \in H^k(S^2 T^{\ast} M) \: ; \: ||h||_{H^k(S^2 T^{\ast} M)} < r \}.$$


\section{Existence and Stability of the Ricci Iteration}
In this section we prove our main result, which establishes existence and stability of the Ricci iterations.
\begin{thm}\label{LSTR}
Suppose $(M,g)$ is a closed $n$-dimensional smooth Einstein manifold with $Ric(g)=g$ such that the following conditions hold:

(a) The kernel of the Lichnerowicz  Laplacian $\Delta_L:L^2(S^2 T^* M)\to L^2(S^2 T^* M)$ is one-dimensional.

(b) The Lichnerowicz Laplacian acting on $L^2(\ker \delta )$ has no eigenvalues lying in $[-2,2]\setminus \{0\}$.

\noindent Then there exists an $\epsilon>0$ such that for each $g_0$ satisfying $g-g_0\in B_{2n+2}(\epsilon)$:

(i) There exists a constant $c>0$ and a Ricci iteration $\{g_i\}_{i=1}^{\infty}$ with $g_1=cg_0$.

(ii) For each $k\ge 2n+2$, $\{g_i\}_{i=I}^{\infty}\subset H^{k}(S^2 T^* M)$ for some $I\in \mathbb{N}$, and $\{g_i\}_{i=I}^{\infty}$ converges to $g$ in 

$H^{k}(S^2 T^* M)$, up to $H^{k}$ diffeomorphisms. 

(iii) If $g-g_0\in B_{2n+2}(\epsilon)$ is smooth, then $\{g_i\}_{i=1}^{\infty}$ is smooth, and its
convergence to $g$ is smooth. 
\end{thm}
\begin{rem}
We have chosen $H^{2n+2}$ to be our base level of regularity in Theorem \ref{LSTR} for convenience. It is conceivable that Theorem \ref{LSTR} remains true by replacing $H^{2n+2}$ with $H^{k}$ for some $k<2n+2$, but we will not dwell on this issue here.
\end{rem}
The main tool behind the proof of Theorem \ref{LSTR} is the Inverse Function Theorem for $C^1$ maps between Banach spaces. However, there are two main obstacles to its immediate application. The first is that the Ricci curvature is scaling invariant, so the inverse Ricci operator is not well-defined. We overcome this problem by modifying the Ricci curvature to account for scaling (see the function $\Phi$ constructed in Subsection \ref{RCM}) as well as introducing the modified Lichnerowicz Laplacian $\tilde{\Delta}_L:\ker(\delta)\to \ker(\delta)$ defined by 
$\tilde{\Delta}_L(h+tg)=-\frac{1}{2}\Delta_L h + 2tg$ for $h$ with $\langle g,h\rangle_{L^2}=0$. Since $\ker \Delta_L=\mathbb{R}\cdot g$ (by condition $(a)$ of Theorem \ref{LSTR}) and $\Delta_L$ respects the decomposition $\ker\delta=\{h\in \ker\delta: \langle g,h\rangle_{L^2}=0\}\oplus \mathbb{R}\cdot g$, we see that $\tilde{\Delta}_L$ is invertible. 
The second obstacle is associated with the diffeomorphism invariance. In particular, the linearisation of the Ricci curvature acts as the identity when applied to the image of $\delta^*$. We overcome this problem by modifying our metric by a diffeomorphism at each step of the iteration.

We outline the proof of Theorem \ref{LSTR}. In Subsection  \ref{RCM}, we consider the problem of inverting Ricci curvature on $\ker \delta$, and construct the inverse Ricci curvature map $\Psi$. Considering only the linearisation of $\Psi$ is not sufficient for the proof of Theorem \ref{LSTR} because of the diffeomorphism invariance of the Ricci curvature, so in Subsection \ref{DM}, we construct another map $\eta$ which modifies a given Riemannian metric by a diffeomorphism. The map $\eta$ is constructed so that $F:=\Psi\circ \eta$ is a map on $\ker\delta$. We then examine the dynamical stability of this map $F$ in Subsection \ref{DSRI}, and conclude the proof of Theorem \ref{LSTR}.  

\subsection{The Ricci Curvature Map}\label{RCM}
We define the modified Ricci curvature map
\begin{align*} \Phi: B_{2n+2}(\epsilon_{2n+2})\cap H^{2n+2}(\ker \delta)\to H^{2n}(S^2 T^{\ast} M),  \hspace{8 mm} h \mapsto \left(1 + \frac{2\langle g,h\rangle_{L^2}}{\langle g,g\rangle_{L^2}}\right)Ric(h + g) - g . \end{align*}
This map is defined and $C^1$ if $\epsilon_{2n+2}$ is small enough. This subsection will be spent finding  $\Phi^{-1}$ and showing it has useful properties. The first lemma is a modification of Lemma 2.6 in \cite{DeTurck85}.

\begin{lem}\label{IRM} Under the assumptions of Theorem \ref{LSTR}, there exists a Banach submanifold $X$ of $H^{2n}(S^2 T^{\ast}M)$, $\epsilon_{2n+2}>0$ and a $C^1$ diffeomorphism $\Psi : X \to H^{2n+2}(\ker \delta) \cap B_{2n+2}(\epsilon_{2n+2})$ such that the following hold:\\
(i) $0 \in X$ and $T_0 X = (\tilde{\Delta}_L-\delta^*\delta G)(H^{2n+2}(\ker\delta))$.\\
(ii) For any $h\in X$,  $\Phi (\Psi(h))=h$. \\
(iii) If $$h+\delta^*(v)\in T_0 X\subset H^{2n}(\ker \delta) \oplus \delta^*(H^{2n+1}T^* M),$$ then  $$d\Psi_0 (h+\delta^*(v))=(\tilde{\Delta}_L)^{-1}(h).$$
\end{lem}

\begin{proof} We can compute the derivative of $\Phi$ at $0$:
$$d\Phi_0(h) = d(Ric)_g(h) + \frac{2\langle g,h\rangle_{L^2}}{\langle g,g\rangle_{L^2}}g = \tilde{\Delta}_L(h)-\delta^*\delta G(h).$$
 Since $\Delta_L:H^{2n+2}(\ker \delta)\to H^{2n}(\ker \delta)$ is a self-adjoint elliptic differential operator with kernel $\R \cdot g$ and image contained in $\{ h \in H^{2n}(\ker \delta ) \: : \: \langle g,h\rangle_{L^2} = 0 \}$, we know $\R \cdot g$ is the orthogonal complement of the image of $\Delta_L$ in $H^{2n}(\ker \delta)$.  Therefore, $d \Phi_0 : H^{2n+2}(\ker \delta ) \to H^{2n} (S^2 T^{\ast} M)$ is a continuous, injective linear map with closed image equal to $(\tilde{\Delta}_L-\delta^*\delta G)(H^{2n+2}(\ker \delta))$. By making $\epsilon_{2n+2}>0$ smaller if necessary, the Inverse Function Theorem then implies that $\Phi(H^{2n+2}(\ker \delta ) \cap B_{2n+2}(\epsilon_{2n+2}))$ is a $C^{\infty}$ Banach submanifold $X$ of $H^{2n}(S^2 T^{\ast} M)$, and $\Phi$ restricted to $H^{2n+2}(\ker \delta )\cap B_{2n+2}(\epsilon_{2n+2})$ is a diffeomorphism onto $X$, with inverse $\Psi$. It is straightforward to verify that $\Psi$ and $X$ satisfy conditions $(i),(ii)$ and $(iii)$.
\end{proof}
We can repeat the proof of Lemma \ref{IRM} for higher levels of regularity, obtaining other maps that satisfy the same properties. However, the following lemma emphasises that we can obtain all of these simply by restricting $\Psi$ to submanifolds of higher regularity. 
\begin{lem}\label{HRORM}
For each $k\ge 2n$, there exists $\epsilon>0$ such that the image of $X_k:=X\cap B_{k}(\epsilon)$ under $\Psi$ is contained in $H^{k+2}(\ker \delta)$, and furthermore, this restricted version of $\Psi$ is a differentiable map from $H^k(S^2 T^* M)$ to $H^{k+2}(\ker(\delta))$. We also have the following: \\
(i) $0 \in X_k$ and $T_0 X_k = (\tilde{\Delta}_L-\delta^*\delta G)(H^{k+2}(\ker\delta))$.\\
(ii) For any $h\in X_k$,  $\Phi (\Psi(h))=h$. \\
(iii) If $$h+\delta^*(v)\in T_0 X_k\subset H^{k}(\ker \delta) \oplus \delta^*(H^{k+1}T^* M),$$ then  $$d\Psi_0 (h+\delta^*(v))=(\tilde{\Delta}_L)^{-1}(h).$$
\end{lem}
\begin{proof}
Since $X_k\subset X$, it follows from $(ii)$ of Lemma \ref{IRM} that $\Phi\circ \Psi$ is the identity map on $X_k$, i.e., condition $(ii)$ holds. Since $\Phi:H^{k+2}(\ker \delta)\to H^{k}(S^2 T^* M)$ is differentiable close to $0$ and $d\Phi_0:H^{k+2}(\ker \delta)\to H^{k}(S^2 T^* M)$ is again a continuous, injective linear map with closed image, the Inverse Function Theorem can be used, at least close to $0$, to find a differentiable inverse for $\Phi:H^{k+2}(\ker \delta)\to H^{k}(S^2 T^* M)$, which must in fact be $\Psi$ by uniqueness. Conditions $(i)$ and $(iii)$ follow. 
\end{proof}
We conclude our examination of the Ricci curvature with the following lemma, which demonstrates that smoothness is also preserved by $\Psi$, whenever it is defined. 
\begin{lem}\label{RPS}
If $h\in X$ is smooth, then $\Psi(h)$ is smooth. 
\end{lem}
\begin{proof}
This follows from local regularity theory of solutions to the prescribed Ricci curvature equation. Indeed, by condition $(ii)$ of Lemma \ref{IRM}, we see that $$\left(1+\frac{2\langle g,\Psi(h)\rangle_{L^2}}{\langle g,g\rangle_{L^2}}\right)Ric(\Psi(h)+g)=g+h.$$ 
By making $\epsilon_{2n+2}$ smaller if necessary, we can assume that $g+h$ is non-degenerate and $1+\frac{2\langle g,\Psi(h)\rangle_{L^2}}{\langle g,g\rangle}\neq 0$. Then $Ric(\Psi(h)+g)$ is smooth and non-degenerate, so Theorem 4.5 (a) of \cite{DeTurck81} implies that $\Psi(h)+g$ is smooth.
\end{proof}

\subsection{The Map which Alters a Metric by a Diffeomorphism}\label{DM}
In this subsection, we develop a discrete analogue of the DeTurck trick with a function $\eta$ which alters a metric by a carefully-chosen diffeomorphism. The construction of $\eta$ itself is carried out in Lemma \ref{EODM} by way of the Implicit Function Theorem. Our use of the Implicit Function Theorem requires that certain actions are sufficiently differentiable, which we verify in  Lemmas \ref{QIS}, \ref{SRSF}, Corollary \ref{COP} and Lemma \ref{MWZAR}. We conclude this subsection with Lemmas \ref{HRODM} and \ref{EPS}, which describe the regularity of $\eta$.  

Our first lemma is a slight modification of some arguments in Section 3 of \cite{Ebin}, and concerns the action of the $H^{k+1}$ diffeomorphism group $\mathcal{D}^{k+1}(M)$ on $H^{k+1}(S^2 T^* M)$ for $k\ge 2n$.
\begin{lem}\label{QIS} The map $Q: \mathcal{D}^{k+1}(M) \times H^{k+1}(S^2 T^{\ast}M) \to H^k(S^2 T^{\ast} M), (\varphi , h) \mapsto \varphi^{\ast}h$ is $C^1$.
\end{lem}
\begin{proof}
Fix $\phi \in \mathcal{D}^{k+1}(M)$, and choose finite covers of $M$ by coordinate charts $(U_{\alpha}, \varphi_{\alpha})_{\alpha = 1}^N$, $(V_{\alpha}, \psi_{\alpha})_{\alpha = 1}^N$ such that $\bar{B}_1(0) \subseteq \varphi_{\alpha}(U_{\alpha})$, $B_{\alpha} := \varphi_{\alpha}^{-1}(B_1(0))$ cover $M$, and $\phi(\bar{B}_{\alpha}) \subseteq V_{\alpha}$.  Since the set of $\zeta \in \mathcal{D}^{k+1}(M)$ satisfying $\zeta(\bar{B}_{\alpha}) \subseteq V_{\alpha}$ is a neighborhood $\mathcal{W}$ of $\phi$ in $\mathcal{D}^{k+1}(M)$, we can find a vector bundle neighborhood $\mathcal{U} \subseteq \mathcal{W}$ of $\phi$ in $\mathcal{D}^{k+1}(M)$ (see \cite{Palais}, Section 12).  In particular, there is a $C^{\infty}$ fiber bundle map $F: TM \to M \times M$ (where we view $M \times M$ as a trivial bundle over $M$) such that $F_{\ast}: H^{k+1}(TM) \to \mathcal{U}, X \mapsto F(X)$ is a diffeomorphism.  It thus suffices to show that
$$H^{k+1}(TM) \times H^{k+1}(S^2 T^{\ast}M) \to H^k(S^2 T^{\ast}M), (X, h) \mapsto F(X)^{\ast}h$$
is $C^1$.

By \cite{Palais}, Section 4, we may view $H^{k+1}(TM)$ as a closed subspace of $\oplus_{\alpha = 1}^N H^{k+1}(\bar{B}_1(0), \R^n)$ via the map
$$X \mapsto \big{(} (\varphi_{\alpha})_{\ast} (X|\bar{B}_{\alpha}) \big{)}_{\alpha = 1}^N.$$
Similarly, we identify $H^{k+1}(S^2 T^{\ast}M)$ as a closed subspace of $\oplus_{\alpha = 1}^N H^{k+1}(\bar{B}_1(0), \R^n \otimes \R^n)$ via
$$h \mapsto \big{(} (\psi_{\alpha}^{-1})^{\ast} h|\bar{B}_1(0) \big{)}_{\alpha = 1}^N ,$$
and we identify $H^k (S^2 T^{\ast}M)$ as a closed subspace of $\oplus_{\alpha = 1}^N H^k (\bar{B}_1(0), \R^n \times \R^n)$ via
$$h \mapsto \big{(} (\varphi_{\alpha}^{-1})^{\ast}|\bar{B}_1(0)) \big{)}_{\alpha = 1}^N .$$
With respect to these identifications, $Q$ is the restriction of the map
$$\oplus_{\alpha = 1}^N H^{k+1}(\bar{B}_1(0), \R^n) \times \oplus_{\alpha = 1}^N H^{k+1}(\bar{B}_1(0), \R^n \otimes \R^n) \to \oplus_{\alpha = 1}^N H^k(\bar{B}_1(0), \R^n \otimes \R^n)$$ 
$$((X_1, ..., X_N), (h_1, ..., h_N)) \mapsto (\tilde{F}_1(X_1)^{\ast}h_1, ..., \tilde{F}_N(X_N)^{\ast} h_N), $$
where $\tilde{F}_{\alpha}(v) := \psi_{\alpha} \circ F((\varphi_{\alpha}^{-1})_{\ast} v) \circ \varphi_{\alpha}^{-1}$ defines a $C^{\infty}$ bundle morphism from $T\bar{B}_1(0) \subseteq T\R^n$ to $\R^n \times \bar{B}_1(0)$.  Thus $Q$ induces a $C^{\infty}$ map $H^{k+1}(\bar{B}_1(0), \R^n) \to \mathcal{D}^{k+1}(\bar{B}_1(0), B_1(0))$ (here 
$\mathcal{D}^k(K, U)$ denotes the set of injective immersions $f\in H^{k}$ with compact domain $K \subseteq \R^n$ and image lying in the open subset $U \subseteq \R^n$).
It therefore suffices to prove that the map
$$\mathcal{D}^{k+1}(\bar{B}_1(0), B_1(0)) \times H^{k+1}(\bar{B}_1(0), \R^n \otimes \R^n) \to H^{k+1}(\bar{B}_1(0), \R^n \otimes \R^n), (\varphi, h) \mapsto \varphi^{\ast} h$$ is $C^1$.  However,  $(\varphi^{\ast} h)(x) = d\varphi(x) \cdot (h \circ \varphi)(x) \cdot d\varphi(x)^t$, where we identify $\R^n \otimes \R^n$ with $n\times n$ $\R$-valued matrices $\R^{n \times n}$. Note that
$$\mathcal{D}^{k+1}(\bar{B}_1(0), B_1(0)) \to H^k(\bar{B}_1(0), \R^{n \times n}), \varphi \mapsto d \varphi $$
is a continuous linear map, so it is smooth and the claim follows from Proposition 2.19 of \cite{IKT}, which implies that $$\mathcal{D}^{k+1}(\bar{B}_1(0), B_1(0)) \times H^{k+1}(\bar{B}_1(0), \R^n \otimes \R^n) \mapsto H^k(\bar{B}_1(0), \R^n \otimes \R^n), (\varphi, f) \mapsto f \circ \varphi$$ 
is $C^1$.
\end{proof}
\noindent 
 We denote by $Iso(M, g)$ the isometry group of $(M, g)$, and by $\mathfrak{iso}(M, g)$ its Lie algebra, consisting of Killing vector fields of $(M, g)$. Since $Iso(M,g)$ is not trivial in general, we need to factor $Iso(M,g)$ out from our diffeomorphism group before we can apply the Implicit Function Theorem. The right coset space  $\mathcal{D}^k(M)/Iso(M, g)$ 
 inherits the structure of a $C^{\infty}$ Banach manifold (constructed in Section 5 of $\cite{Ebin}$) such that the quotient map $\pi_k: \mathcal{D}^k(M) \to \mathcal{D}^k(M) / Iso(M, g)$ is a smooth submersion.  Note that composition with isometries preserves each $\mathcal{D}^l(M)$ for $l>k$, so we have a (noncontinuous) inclusion
$$\mathcal{D}^l(M) / Iso(M, g) \to \mathcal{D}^k(M) / Iso(M, g) $$
for $l > k$, and we can view the former as a (nonclosed) subset of the latter. The following lemma produces a section of the quotient map which respects these inclusions. 
\begin{lem}\label{SRSF} There exists a neighborhood $U$ of the identity coset in $\mathcal{D}^{2n+1}(M)/Iso(M, g)$ and a smooth section $\sigma: U \to \mathcal{D}^{2n+1}(M)$ of the quotient map $\pi_{2n+1}: \mathcal{D}^{2n+1} \to \mathcal{D}^{2n+1}/Iso(M, g)$ such that, for any integer $k \ge 2n+1$, the restriction of $\sigma$ to $\mathcal{D}^k(M)/Iso(M, g)$ is a smooth section of $\pi_k$. We also have $\sigma([e])=e$, where $e$ is the identity diffeomorphism. 
\end{lem}
\begin{proof} For all $k \ge 2n+1$, define $Y_k := \{ v \in H^k(T^{\ast}M) \: : \: \langle v, X \rangle_{L^2(M, g)} = 0 \mbox{ for all } X \in \mathfrak{iso}(M, g) \}$, and 
$$\Sigma_k : \mathfrak{iso}(M, g) \times \{v \in Y_k \: : \: ||v||_{H^k} < \epsilon\} \to \mathcal{D}^k(M), (X, v) \mapsto \Big{(}p \mapsto (\mbox{Exp} X) \circ \exp(v(p))\Big{)},$$
where $\epsilon >0$ is chosen (via the Sobolev embedding) so that $||v||_{L^{\infty}} < \mbox{inj}(M, g)$ whenever $||v||_{H^k} < \epsilon$, $\mbox{Exp}: \mathfrak{iso}(M, g) \to Iso(M, g)$ is the Lie exponential map, and $\exp: TM \to M$ is the Riemannian exponential map of the Einstein metric.  By Lemma 5.5 of \cite{Ebin} and the definition of the $C^{\infty}$ structure of $\mathcal{D}^k(M)$ via vector bundle neighborhoods (see, for example, Section 13 of \cite{Palais}), $\Sigma_k$ is smooth.  Since $(d \Sigma_k)_{(0, 0)}: \mathfrak{iso}(M, g) \times Y_k \to H^k(T^{\ast} M)$ is a linear isomorphism (its restriction to each factor is just the corresponding inclusion map), we obtain that $\Sigma_k$ restricts to a diffeomorphism from a neighborhood of $(0, 0)$ in $\mathfrak{iso}(M, g) \times Y_k$ to a neighborhood of the identity in $\mathcal{D}^k(M)$.  Moreover, note that $d_1 \Sigma_k$ (the differential in the first factor) is a local bundle isomorphism from $T_1(\mathfrak{iso}(M, g) \times Y_k)$ to the subbundle $\cup_{\eta \in \mathcal{D}^k(M)} T_{\eta}(Iso(M, g)\circ \eta)$, where $T_1(M \times N)$ denotes the subbundle $\cup_{(p, q) \in M \times N} T_p M $ for any manifolds $M, N$.  By Lemma 5.9 of \cite{Ebin}, we conclude that $\pi_k \circ \Sigma_k |Y_k$ is a diffeomorphism from a neighborhood of $0$ to a neighborhood of the identity coset.  Take 
$$\sigma_k := \Sigma_k \circ (\pi_k \circ \Sigma_k|Y_k)^{-1},$$
and note that $\sigma_k |(\mathcal{D}^k(M) /Iso(M, g)) = \sigma_l$ for $l > k$, since the definition of $\Sigma_k$ does not depend on $k$ other than the choice of domains, and since $Y_k \cap H^l(T^{\ast} M) = Y_l$.  The lemma follows by taking $\sigma := \sigma_{2n+1}$.
\end{proof}
\begin{cor}\label{COP}
The map $P:U\times H^{2n+1}(\ker \delta)\to H^{2n}(S^2 T^* M)$ with $P([\varphi],h)=\sigma([\varphi])^*(h+g)-g$ is $C^1$ and satisfies $P([e],0)=0$. Furthermore, when restricted to $ (U\cap \mathcal{D}^{k+1}(M)/Iso(M,g))\times H^{k+1}(\ker \delta)$, $P$ is a $C^1$ map to $H^{k}(S^2 T^* M)$ for any $k\ge 2n$. 
\end{cor}
\begin{proof}
We see that $P([\varphi],h)=Q(\sigma([\varphi]),g+h)-g$. By Lemmas \ref{QIS} and \ref{SRSF}, we know that $Q$ and $\sigma$ are $C^1$, so $P$ is also $C^1$. Lemmas \ref{QIS} and \ref{SRSF} also imply that $Q$ and $\sigma$ are $C^1$ when restricted to higher levels of regularity, so the same is true for $P$. It is clear that $P([e],0)=0$. 
\end{proof}
In the next lemma, we construct a function $\xi$ which measures distance from $\Phi(\ker \delta)$. We use this function $\xi$ in Lemma \ref{EODM}, where we use the Implicit Function Theorem to construct a map $\eta$ that sends $\ker\delta$ to $\Phi(\ker \delta)$ via a diffeomorphism. 
\begin{lem}\label{MWZAR}
There exists a $C^1$ function $\xi:H^{2n}(S^2 T^* M)\to \delta^* (H^{2n+1}(T^* M))$ defined close to $0$ such that:\\
(i) $\xi(h)=0$ if and only if $h\in X$, and $\ker(d\xi_0)=(\tilde{\Delta}_L-\delta^* \delta G)(H^{2n+2}(\ker \delta))$.\\
(ii) For any $k\ge 2n$, there exists an $\epsilon>0$ such that $\xi(H^k(S^2 T^* M)\cap B_k(\epsilon))\subseteq \delta^* (H^{k+1}(T^* M))$, and this restricted version of $\xi$ is  differentiable when considered as a map to $\delta^* (H^{k+1}(T^* M))$. Those $h\in H^k(S^2 T^* M)\cap B_k(\epsilon)$ satisfying $\xi(h)=0$ are precisely those $h$ lying in $X_k$. Furthermore, $\ker(d\xi_0)=(\tilde{\Delta}_L-\delta^* \delta G)(H^{k+2}(\ker \delta))$.
\end{lem}
\begin{proof}
The map $\Phi_{\ker \delta}:=proj_{\ker \delta}\circ \Phi:H^{2n+2}(\ker \delta)\to H^{2n}(\ker \delta)$ is differentiable at $0$ and  has a differentiable inverse around $0$ by the Inverse Function Theorem. Consider the  function $\xi: H^{2n}(S^2 T^* M)=H^{2n}(\ker \delta)\oplus \delta^*(H^{2n+1}T^* M)\to \delta^{\ast}(H^{2n+1}T^{\ast}M)$ with \begin{align*}
    \xi(h,v)=v-proj_{\delta^{\ast}(H^{2n+1}(T^{\ast}M)) }\circ \Phi\circ (\Phi_{\ker \delta})^{-1}h.
\end{align*}
Close to $0$, this map is defined and differentiable, and 
\begin{align*}
   d\xi_0(h,v)&=v-proj_{\delta^{\ast}(H^{2n+1}(T^{\ast}M)) }\circ (\tilde{\Delta}_L-\delta^* \delta G)\circ (\tilde{\Delta}_L)^{-1}h\\
   &=v+\delta^* \delta G (\tilde{\Delta}_L)^{-1}h.
\end{align*} Furthermore, $\xi(h,v)=0$ if and only if $$(h,v)=(\Phi_{\ker \delta}\circ (\Phi_{\ker \delta})^{-1}h,proj_{\delta^{\ast}(H^{2n+1}(T^{\ast}M)) }\circ \Phi\circ (\Phi_{\ker \delta})^{-1}h)=\Phi ((\Phi_{\ker \delta})^{-1}h),$$
which is equivalent to $(h,v)\in X$. This establishes $(i)$. 

To show that $(ii)$ holds, note that when restricted to $H^{k}(\ker \delta)$,  $(\Phi_{\ker \delta})^{-1}$ is an inverse for the differentiable function $\Phi_{\ker \delta}:H^{k+2}(\ker \delta)\to H^{k}(\ker \delta)$. Therefore, the Inverse Function Theorem implies that $(\Phi_{\ker \delta})^{-1}:H^{k}(\ker \delta)\to H^{k+2}(\ker \delta)$ is differentiable close to $0$. We can then repeat the first part of the proof with this new version of $\Phi_{\ker \delta}$ to conclude that $(ii)$ holds. 
\end{proof}

The next lemma constructs a map $\eta$ which alters a metric by a diffeomorphism, and the remaining lemmas in this section describe some of the useful properties of this map. 
\begin{lem}\label{EODM}
There exists a $C^1$ function $\eta: B_{2n+1}(\epsilon_{2n+1}) \cap H^{2n+1}(\ker \delta ) \to X$ defined for small $\epsilon_{2n+1}$ such that:\\
(i) For any $h \in B_{2n+1}(\epsilon_{2n+1}) \cap H^{2n+1}(\ker \delta )$, there is a diffeomorphism $\varphi \in \mathcal{D}^{2n+1}(M)$ such that $\eta(h)+g = \varphi^{\ast}(h+ g)$.\\ 
(ii) $d(\eta)_{0}(h)=h-\delta^* \delta G (\tilde{\Delta}_L^{-1}h)$ for each $h\in H^{2n+1}(\ker \delta )$.\\
\end{lem}
\begin{proof}
Consider the map
$$\xi \circ P : U\times H^{2n+1}(\ker \delta) \to \delta^{\ast}(H^{2n+1}(T^{\ast}M)).$$ By Corollary \ref{COP} and Lemma \ref{MWZAR}, this composition map 
is $C^1$ and satisfies $(\xi \circ P)([e],0) = 0$.  Moreover, for any $[v] \in T_{[e]}(\mathcal{D}^{2n+1}(M)/Iso(M, g)) \equiv H^{2n+1}(T^{\ast}M)/\mathfrak{iso}(M, g)$, we have 
$$d(\xi \circ P)_{([e],0)} ([v]) = d\xi_0 (\delta^{\ast}(v)).$$ 
However, we know by 
$(i)$ of Lemma \ref{MWZAR} that $\ker(d\xi_0) = (\tilde{\Delta}_L-\delta^*\delta G)(H^{2n+2}(\ker \delta ))$, so $d_1(\xi \circ P)_{([e],0)}$ is a Banach isomorphism from $H^{2n+1}(T^{\ast}M)/\mathfrak{iso}(M, g)$ to $\delta^{\ast}(H^{2n+1}(T^{\ast}M))$.  Thus the Implicit Function Theorem gives a $C^1$ map $\alpha: H^{2n+1}(\ker \delta) \cap B_{2n+1}(\epsilon_{2n+1}) \to U$ such that 
$\xi\circ P(\alpha(h),h))=0$ for all $h \in H^{2n+1}(\ker \delta ) \cap B_{2n+1}(\epsilon_{2n+1})$, i.e., 
$\sigma(\alpha(h))^{\ast}(g + h) - g \in X$.

Define
$$\eta: B_{2n+1}(\epsilon_{2n+1}) \cap H^{2n+1}(\ker \delta ) \to X, \: h \mapsto (\sigma(\alpha(h))^{\ast}(g + h) - g.$$
Then $\eta$ is a $C^1$ map satisfying $(i)$.  Also, since $d\eta_0(h)=d(\sigma(\alpha))_0h+h\in (\tilde{\Delta}_L-\delta^*\delta G)(H^{2n+2}(\ker \delta))$ and $d(\sigma(\alpha))_0h\in \delta^*(H^{2n+1}T^* M)$ for each $h\in H^{2n+1}(\ker \delta)$, we find that $d(\sigma(\alpha))_0h=-\delta^* \delta G (\tilde{\Delta}_L^{-1}h)$, so $(ii)$ is also satisfied. 
\end{proof}
We now fix once and for all $\eta$ satisfying the conclusion of Lemma \ref{EODM}. 
\begin{lem}\label{HRODM}
If we restrict $\eta$ to $H^{k+1}(\ker \delta)$ for some $k\ge 2n$, then close to $0$, $\eta$ is a map to $X_k$ which is differentiable. We also have that:\\
(i) For any $h\in H^{k+1}(\ker \delta)$ such that $\eta(h)$ exists, there is a diffeomorphism $\varphi\in \mathcal{D}^{k+1}(M)$ such that $\eta(h)+g=\varphi^*(h+g)$. \\
(ii) $d(\eta)_0(h)=h-\delta^* \delta G(\tilde{\Delta}_L^{-1}h)$. 
\end{lem}
\begin{proof}
As in the proof of Lemma \ref{EODM}, construct the function $\alpha:H^{2n+1}(\ker \delta)\to U$. We see that $\xi\circ P(\alpha(h),h)=0$ for all $h\in H^{k+1}(\ker \delta)$, at least when this expression is defined. Now note that $\xi\circ P$ is differentiable around $([e],0)$ when considered a map from $(U\cap \mathcal{D}^{k+1}(M)/Iso(M,g))\times H^{k+1}(\ker \delta) $ to $\delta^*(H^{k+1}T^* M)$, because of Corollary \ref{COP} and Lemma \ref{MWZAR}. Furthermore, $d_1(\xi\circ P)_{([e],0)}$ is again a Banach isomorphism. Therefore, the Implicit Function Theorem implies that $\alpha:H^{k+1}(\ker \delta)\to U\cap (\mathcal{D}^{k+1}(M)/Iso(M,g))$ is differentiable close to $0$. The restriction of $\eta$ to $H^{k+1}$ is obtained by restricting $\alpha$ to $H^{k+1}$, so conditions $(i)$ and $(ii)$ readily follow. 
\end{proof}
\begin{lem}\label{EPS}
If $h\in B_{2n+1}(\epsilon_{2n+1})\cap H^{2n+1}(\ker \delta)$ is smooth, then $\eta(h)$ is also smooth. 
\end{lem}
\begin{proof}
We see that $\eta(h)+g=(1+2\frac{\langle g,k\rangle_{L^2}}{\langle g,k\rangle_{L^2}})Ric(g+k)$ for some $k\in H^{2n+2}(\ker \delta)$ since $\eta(h)\in X$. Furthermore, $\eta(h)+g=\varphi^*(g+h)$ for some diffeomorphism $\varphi$ by $(ii)$ of Lemma \ref{EODM}. Then 
\begin{align*}
    g+h=(1+2\frac{\langle g,k\rangle_{L^2}}{\langle g,k\rangle_{L^2}})Ric((\varphi^{-1})^{*}(g+k)).
\end{align*}
Similarly to the proof of Lemma \ref{RPS}, Theorem 4.5 of \cite{DeTurck81} can be used to demonstrate that $K:=(\varphi^{-1})^*(g+k)$ is smooth. We also see that 
\begin{equation}\label{EFVS}
    \delta_g(\varphi^* K)=0.
\end{equation} In local co-ordinates, \eqref{EFVS} is a second order quasi-linear equation in $\varphi$:
\begin{align*}0 &= g^{ik} (K_{pq} \circ \varphi) \dfrac{\partial \varphi^p}{\partial x^j} \dfrac{\partial^2 \varphi^q}{\partial x^i \partial x^k} + g^{ik} (K_{pq} \circ \varphi) \dfrac{\partial \varphi^q}{\partial x^k} \dfrac{\partial^2 \varphi^p}{\partial x^i \partial x^j} + g^{ik} \dfrac{\partial \varphi^p}{\partial x^j} \dfrac{\partial \varphi^q}{\partial x^k} \dfrac{\partial ( K_{pq} \circ \varphi)}{\partial x^i} \\ &- g^{ik}\left( \dfrac{\partial \varphi^p}{\partial x^r} \dfrac{\partial \varphi^q}{\partial x^k} \Gamma_{ij}^r + \dfrac{\partial \varphi^p}{\partial x^j} \dfrac{\partial \varphi^q}{\partial x^r} \Gamma_{ik}^r \right)K_{pq} \circ \varphi, \hspace{6 mm} j=1,\cdots,n,
\end{align*} 
where $\Gamma_{ij}^k$ are the Christoffel symbols of $g$. 
This equation is strongly elliptic because $\varphi$ is close to the identity in $\mathcal{D}^{2n+1}(M)$, and $K$ is close to the positive definite Einstein metric $g$ in $H^{2n+1}(S^2 T^* M)$. Since $\varphi\in C^{2,\alpha}$ for some $\alpha\in (0,1)$ (by Sobolev embedding), standard elliptic regularity arguments can be applied to demonstrate that $\varphi$ is automatically smooth. For example, we can apply Theorem 10.3.1 (b) of \cite{Nico} to demonstrate by induction that $\varphi\in C^{k,\alpha}$ for each $k\ge 2$. 
\end{proof}
\subsection{Dynamical Stability of the Ricci Iteration}\label{DSRI}
Consider the map $F:H^{2n+1}(\ker \delta )\to H^{2n+1}(\ker \delta )$ defined by 
\begin{align*}
    F=\Psi \circ \eta.
\end{align*}
\begin{lem}\label{FIAC}
For each $k\ge 2n+1$, the induced map
$F:H^{k}(\ker \delta)\to H^{k+1}(\ker \delta)$ is continuous at $0$. Furthermore, $F:H^{k}(\ker \delta)\to H^{k}(\ker \delta)$  is a local contraction around $0$. 
\end{lem}
\begin{proof}
Lemmas \ref{HRORM} and \ref{HRODM} imply that the composition map $F:H^{k}(\ker \delta)\to H^{k+1}(\ker \delta)$ is a differentiable map close to $0$, so is also continuous.  Lemmas \ref{HRORM} and \ref{HRODM} also allow us to compute the derivative of $F:H^{k}(\ker \delta)\to H^{k+1}(\ker \delta)$ by the chain rule, giving $d(F)_0(h)=(\tilde{\Delta}_L)^{-1}(h)$ for $h\in H^{k}(\ker \delta)$. By the assumptions on the spectrum of the Lichnerowicz Laplacian in the statement of Theorem \ref{LSTR}, $d(F)_0(h)$ is a strict contraction when considered as a linear map on $H^{k}(\ker \delta)$. Therefore, by continuity of the derivative, $F:H^k(\ker \delta)\to H^k(\ker \delta)$ is also a contraction, at least close to $0$. 
\end{proof}

\begin{proof}[Proof of Theorem \ref{LSTR}]
Take an initial metric $g_1\in H^{2n+2}(S^2 T^* M)$ which is close to $g$. By taking a $H^{2n+2}$ diffeomorphism if necessary, we can assume that $g_1\in H^{2n+1}(\ker \delta)$. Indeed, by constructing the differentiable map 
$proj_{\delta^* (H^{2n+2}(T^* M) )}\circ P: (U\cap \mathcal{D}^{2n+2}(M)/Iso(M,g))\times H^{2n+2}(S^2 T^* M)\to \delta^* (H^{2n+2}(T^* M) )$, noting that the derivative of the first component is a Banach isomorphism, and applying the Implicit Function Theorem, we can find the required diffeomorphism. If $g_1$ is smooth, the diffeomorphism will also be smooth by arguments similar to those appearing in the proof of Lemma \ref{EPS}.

Define the sequence of Riemannian metrics $g_i$ by $g_{i+1}-g=F(g_i-g)$. By Lemma \ref{FIAC}, we can prove inductively that $\left|\left|g_i-g\right|\right|_{H^k(S^2 T^* M)}$ converges to $0$ for each $k\ge 2n+1$. For each $i$, $g_{i+1}-g=\Psi\circ \eta(g_i-g)$, so 
\begin{equation}\label{FRI}
g_{i+1}=\varphi_i^* (\Psi(g_i-g)+g)
\end{equation} 
for some diffeomorphism $\varphi_i\in \mathcal{D}^{2n+1}(M)$, although we note that $g_i-g$ becomes small in each $H^k(S^2 T^* M)$, so Lemma \ref{HRODM} implies that the regularity of $\varphi_i$ is increasing. By taking the Ricci curvature of both sides of \eqref{FRI}, we find 
\begin{align*}
Ric(g_{i+1})=\frac{1}{1+2\frac{\langle g,\Psi(g_i-g)\rangle_{L^2}}{\langle g,g\rangle_{L^2}}}\varphi_i^*g_i.
\end{align*}
Therefore, the rescaled metrics $\tilde{g}_i=\frac{g_i}{1+2\frac{\langle g,\Psi(g_i-g)\rangle_{L^2}}{\langle g,g\rangle_{L^2}}}$ form a Ricci iteration, up to diffeomorphisms, and converge to $g$ in every $H^k$ norm since $\Psi(g_i-g)\to 0$. Also note that $\tilde{g}_1$ only differs from $g_1$ by scaling. 

Now if $g_1$ is smooth, then Lemmas \ref{RPS} and \ref{EPS} imply that $g_i$ is smooth for all $i\ge 1$, so our convergence is smooth as well.  
\end{proof}
The following lemma simplifies the task of deciding whether or not a given Einstein manifold satisfies the conditions of Theorem \ref{LSTR} by relating the spectrum of $\Delta_L:\ker(\delta)\to \ker (\delta)$ to the spectrums of
$\Delta_L:S^2 (T^* M)\to S^2 (T^* M)$ and the Lichnerowicz Laplacian on one-forms $\Delta_L:T^* M\to T^* M$, given by $\Delta_L v=\Delta v-v$ if $Ric(g)=g$. 
\begin{lem}\label{COELL}
Suppose $\lambda$ is an eigenvalue of $\Delta_L:S^2 T^* M\to S^2 T^* M$ with multiplicity $m$, and let $w\in S^2 T^* M$ satisfy $\Delta_Lw=\lambda w$. Choose $p=1$ if $\lambda=-2$ and $p=0$ otherwise. If $\lambda$ is also an eigenvalue of  $\Delta_L:T^* M\to T^* M$ with multiplicity $m+p \dim (\ker \delta^*)$, then $w\in \delta^* (T^* M)$. 
\end{lem}
\begin{proof}
Suppose that $\lambda\neq -2$ is an eigenvalue of $\Delta_L:S^2 T^* M\to S^2 T^* M$ with corresponding $m$-dimensional vector space of eigentensors $W_{\lambda}$. By assumption, there is an $m$-dimensional vector space $V_{\lambda}\subset T^* M$ such that for each $v\in V_{\lambda}$, $\Delta_L v=\lambda v$. For each non-zero $v\in V_{\lambda}$, $\delta^*(v)\neq 0$ because $2\delta G\delta^*(v)=-\Delta_L v-2v\neq 0$. Therefore, $\delta^*(V_{\lambda})$ is an $m$-dimensional subspace of $S^2 T^* M$. Since the Lichnerowicz Laplacian commutes with $\delta^*$, we then find that $\delta^*(V_{\lambda})\subseteq W_{\lambda}$, but since the spaces have the same dimension, 
$\delta^*(V_{\lambda})= W_{\lambda}$.

If $\lambda=-2$, then $V_{\lambda}$ is $m+\dim (\ker\delta^*)$ dimensional by assumption. Since $V_{\lambda}$ contains $\ker \delta^*$, we again find that $\delta^* (V_{\lambda})$ is an $m$-dimensional subspace of $W_{\lambda}$, so the two spaces are in fact identical. 
\end{proof}

\section{The Lichnerowicz Laplacian on Symmetric Spaces}\label{LLSS}
Lemma \ref{COELL} demonstrates that we can determine whether a given Einstein manifold satisfies the conditions of Theorem \ref{LSTR} by knowing the following:
\begin{enumerate}
    \item The dimension of the vector space of Killing fields on $(M,g)$. 
    \item The eigenvalues (including multiplicites) of the Lichnerowicz Laplacian on $T^*M=S^1T^* M$.
    \item The eigenvalues (including multiplicites) of the Lichnerowicz Laplacian on $S^2 T^* M$.
\end{enumerate}
Such knowledge is possible, for example, if our Einstein manifold is a compact symmetric space $(M = G/K,g)$, with $g$ induced by $-B$, where $B$ is the Killing form (although, the Einstein constant will not be $1$ in general, so we will need to rescale eventually). 
Indeed, $G$ is the connected component of the isometry group of the Einstein manifold containing the identity, so the dimension of the vector space of Killing fields is the dimension of $G$. 

On the other hand, to compute the spectrum of the two Lichnerowicz Laplacians, one can use representation theory. This has been done for various symmetric spaces in, for example, \cite{IT,Koiso80}. 
To start, note that for $p=1,2$, 
\begin{align*}
    L^2(S^pT^*M)=L^2(EM),
\end{align*}
where $E=S^p((\mathfrak{g}/\mathfrak{k})^*)$ is a $K$ representation with the adjoint action $Ad$, 
$EM=G\times _{Ad}E$ is a vector bundle, and $L^2(EM)$ are $L^2$  sections of $EM$. These $L^2$ sections correspond to functions $f : G \to E$ which satisfy $f(gk) = Ad(k)^{-1}f(g)$ for all $g \in G$ and $k \in K$. 

For any unitary $K$-representation $(V, \rho)$, we may define the induced $G$-representation
$$Ind_K^G(V) := \{ f \in L^2(G; V) \: ; \: f(gk) = \rho(k)^{-1}f(g) \mbox{ for all } k \in K, g \in G \}.$$
With this, we can see that $L^2(EM)$ is just the induced representation $Ind_K^G(E)$.  By Frobenius reciprocity for unitary representations of compact Lie groups (see, for example, Chapter 6 of \cite{Folland}), the induced representation $Ind_K^G(E)$ admits the following isotypic decomposition:
$$Ind_K^G(E)= \widehat{\bigoplus}_{\lambda\in \hat{G}}\Gamma_{\lambda} \otimes Hom_K(Res^G_K\Gamma_{\lambda},E),$$ where $\Gamma_{\lambda}$ is an irreducible $G$-representation with highest weight $\lambda$, and $G$ acts as $\rho_{\lambda} \otimes 1$ on each summand. It is well-known that the Casimir element $\mathcal{C}$ acts on each summand as $(\lambda+2\rho,\lambda)\otimes 1$, where $(\cdot,\cdot)$ is an inner product on the root space, and $\rho$ is the half sum of positive roots. This observation allows us to compute the spectrum of $\Delta_L$ acting on $L^2(EM)$ because of the following. 
\begin{lem}For $E=(\mathfrak{g}/\mathfrak{k})^*$ or $E=S^2 ((\mathfrak{g}/\mathfrak{k})^*)$, we have $\mathcal{C} = -\Delta_L$. 
\end{lem}
\begin{proof}
The proof is almost identical to the proof of Proposition 5.3 in \cite{Koiso80}. The only difference is that the tensor fields that we are applying $\Delta_L$ and $C$ to are all symmetric. 
\end{proof}
Thus, in order to compute eigenvalues of $\Delta_L$ and their multiplicities, it suffices to decompose $Res_K^G \Gamma_{\lambda}$ and $E$ into irreducible $K$-representations.  The decomposition of $Res_K^G \Gamma_{\lambda}$ can be computed using the branching rules of irreducible $G$-representations into irreducible $K$-representations, which are well-understood for many pairs $(G, K)$ (see, e.g., Chapter 8 of \cite{GoodmanWallach}). We encounter an example of the use of this method in the next section. 
\section{Examples}
In this section, we consider three examples of closed Einstein manifolds with positive Ricci curvature, and discuss their stability in the sense of Theorem \ref{LSTR}. 
\subsection{Even-Dimensional Spheres}
We will use the method discussed in Section \ref{LLSS} in order determine the stability of the even dimensional spheres. To that end, let 
$G=SO(2n+1)$ and $K=SO(2n)$, and consider the sphere $\mathbb{S}^{2n}=G/K$ equipped with the  Riemannian metric induced by $-B$. We let $\mathfrak{g}=\mathfrak{so}(2n+1)$ and $\mathfrak{k}=\mathfrak{so}(2n)$. The Cartan subalgebra $\mathfrak{h}$ of $\mathfrak{g}$ is $n$-dimensional and an irreducible $\mathfrak{g}$ representation is determined by a highest weight 
$\lambda=(\lambda_1,\cdots,\lambda_{n})$, consisting of integers $\lambda_i$ satisfying 
\begin{align*}
    \lambda_1\ge \lambda_2\ge \cdots \ge \lambda_n.
\end{align*} 
Now we note that the Cartan subalgebra of $\mathfrak{so}(2n)$ is also $n$-dimensional, so an irreducible $\mathfrak{k}$ representation is determined by a highest weight $\bar{\lambda}$ which has the same number of components as for $\mathfrak{so}(2n+1)$, and the branching law is given as
\begin{align*}
    Res_{\mathfrak{so}(2n)}^{\mathfrak{so}(2n+1)}(\Gamma_{\lambda})=\bigoplus \Gamma_{\bar{\lambda}},
\end{align*}
with the sum taken over $\bar{\lambda}$ satisfying 
\begin{align*}
    \lambda_1\ge \bar{\lambda}_1\ge \lambda_2\ge \bar{\lambda}_2\ge \cdots \ge \bar{\lambda}_{n-1}\ge \lambda_n\ge \left|\bar{\lambda}_{n}\right|.
\end{align*}
We now compute $Hom_{K}(Res_K^G \Gamma_{\lambda},E)$ for $E=S^2((\mathfrak{g}/\mathfrak{k})^*)$. We first note that $E$ decomposes into the direct sum of two irreducible $K$ subrepresentations, namely the scalar multiples of the Killing form, and the traceless tensors. These are $K$-irreducible representations having highest weights $(0,0,\cdots,0)$ and $(2,0,0,\cdots,0)$ respectively. Therefore,  $Hom_{K}(Res_K^G \Gamma_{\lambda},E)$ is only non-trivial when $\lambda=(\lambda_1,\lambda_2,0,0,\cdots,0)$ with $\lambda_1\ge \lambda_2$ and $\lambda_2\le 2$, in which case, $Hom_{K}(Res_K^G \Gamma_{\lambda},E)$ is $1$-dimensional by Schur's Lemma. The eigenvalues of the Casimir acting on such a $\Gamma_{\lambda}$ 
 are $(\lambda_1+\lambda_2)(\lambda_1+2n-1-\lambda_2)$. 

We now also compute $Hom_{K}(Res_K^G \Gamma_{\lambda},E)$ for $E=(\mathfrak{g}/\mathfrak{k})^*$. In this case, $E$ is $K$-irreducible with highest weight $(1,0,0,\cdots,0)$. Therefore,  $Hom_{K}(Res_K^G \Gamma_{\lambda},E)$ is only non-trivial when $\lambda=(\lambda_1,0,0,0,\cdots,0)$ with $\lambda_1\ge 1$. In this case, $Hom_{K}(Res_K^G \Gamma_{\lambda},E)$ is again $1$-dimensional, and the corresponding eigenvalues of the Casimir are $\lambda_1(\lambda_1+2n-1)$. 

Now suppose that we rescale the Killing form so that $Ric(g)=g$, and the new Lichnerowicz Laplacian has an eigenvalue in $[-2,2]$. Then since the metric induced by the Killing form has Ricci curvature $2n-1$, we obtain an eigenvalue $\lambda\in [-2(2n-1),2(2n-1)]$ of the Casimir $\mathcal{C}$ acting on $L^2(EM)$ when $E=S^2((\mathfrak{g}/\mathfrak{k})^*)$. 
Then $\lambda=0$ or $\lambda=-2n$. The eigenvalue $\lambda=0$ has multiplicity equal to the dimension of $\Gamma_{(0,0,\cdots,0)}$, which is one. The eigenvalue $\lambda=-2n$ has multiplicity identical to the dimension of $\Gamma_{(1,0,0,\cdots,0)}$. However, this same eigenvalue $\lambda=-2n$ occurs for the Casimir acting on $L^2 (EM)$ when $E=(\mathfrak{g}/\mathfrak{k})^*$, and with the same multiplicity. Thus, by Lemma \ref{COELL} the even-dimensional spheres satisfy the conditions of Theorem \ref{LSTR}. Therefore, for any Riemannian metric $g_0$ close to the Einstein metric $g$ in $H^{2n+2}$, there is a Ricci iteration $\{g_i\}_{i=1}^{\infty}$ such that $g_1=cg_0$ for some $c>0$, and for each $k\ge 2n+2$, there exists $I\in \mathbb{N}$ such that $\{g_i\}_{i=I}^{\infty}$ is a sequence in $H^k$ which is convergent to $g$.
\begin{rem}
For the sphere $\mathbb{S}^n$, the eigenvalues and their multiplicities can alternatively be computed by embedding the sphere into $\mathbb{R}^{n+1}$ and using harmonic polynomials. This is done, for example, in \cite{Boucetta09}.
\end{rem}
\subsection{Complex Projective Space}
Rescale the Fubini-Study metric on $\mathbb{CP}^n$ so that we arrive at a Riemannian metric $g$ satisfying $Ric(g)=g$. By using either the method from Section \ref{LLSS} or results in \cite{Boucetta10}, we find that the Lichnerowicz Laplacian has a one-dimensional kernel, but also has an eigenvalue of $-2$ with corresponding eigentensors in $\ker\delta$. As a consequence, we cannot apply Theorem \ref{LSTR}. We can repeat the majority of the proof of Theorem \ref{LSTR} to arrive at the `Ricci iterations up to diffeomorphism' map $F:H^{2n+1}(\ker \delta)\to H^{2n+1}(\ker \delta)$. However, we can no longer conclude that $F$ is a local contraction around $0$. We conjecture that in general, the map $F$ is not stable close to $0$ because we encounter instability of the closely-related Ricci flow near the Fubini-Study metric, as seen in \cite{KnoNat,Kro13}. However, it should be noted that when restricting to K\"ahler metrics on $\mathbb{CP}^n$, the Fubini-Study metric is stable under the Ricci iteration, as shown in \cite{DarvasRubinstein19}. In fact, the authors demonstrate that a Ricci iteration starting from any K\"ahler metric converges to the Fubini-Study metric. 

\subsection{The Jensen Metric}
As seen in \cite{Jensen}, there are two homogeneous Einstein metrics on the homogeneous space $\mathbb{S}^{4n+3} =
G/K$, where $G = Sp(n + 1) \times Sp(1)$ and $K = Sp(n) \times Sp(1)$. One of these Einstein metrics is the round metric of
constant sectional curvature. As demonstrated in Theorem 3.2 (b) of \cite{BPRZ18}, there exists a sequence
of $G$-invariant Riemannian metrics $\{g_i \}^{\infty}_{i=-\infty}$ satisfying $g_i = Ric(g_{i+1})$, such that as $i \to \infty$, $g_i$ converges to
the round Einstein metric, but as $i \to -\infty$, $g_i$ converges to the non-round Einstein metric. As a
consequence, even if the `Ricci iterations up to diffeomorphism' map $F$ can be constructed around this non-round Einstein metric (which will be possible if the kernel of the Lichnerowicz Laplacian is one-dimensional), it will not be stable. 

\section{Acknowledgements}
We would like to thank Artem Pulemotov and Xiaodong Cao, our respective PhD supervisors, for their supervision of this project. 
We are also grateful to Andr\'es Fern\'andez Herrero for helpful discussions on Frobenius reciprocity. 
Timothy Buttsworth was supported in this research by an Australian Government Research Training Program
Scholarship.

\end{document}